\documentclass[runningheads]{llncs}
\usepackage[T1]{fontenc}
\usepackage{graphicx}
\usepackage{mathtools, amsmath, amssymb, xspace}
\usepackage{xcolor}
\usepackage{hyperref, varioref, cleveref}
\newcommand{\GG}{\ensuremath{\mathcal G}\xspace}
\newcommand{\VV}{\ensuremath{\mathcal V}\xspace}
\newcommand{\EE}{\ensuremath{\mathcal E}\xspace}

\begin{document}
\title{A Tight Lower Bound on Cubic Vertices and Upper Bounds on Thin and Non-thin edges in Planar Braces}
\titlerunning{On Cubic Vertices and Thin Edges of Planar Braces}
%
\author{Koustav De\inst{1}\orcidID{0000-0002-4434-0627}}
\authorrunning{K. De}
%
\institute{Department of Computer Science and Engineering, IIT Kharagpur, Kharagpur, West Midnapore, West Bengal - 721302, India \\
\email{koustavde7@kgpian.iitkgp.ac.in}}

\maketitle              
\begin{abstract}
For a subset $X$ of the vertex set $\VV(\GG)$ of a graph $\GG$, we denote the set of edges of $\GG$ which have exactly one end in $X$ by $\partial(X)$ and refer to it as the cut of $X$ or edge cut $\partial(X)$. A graph $\GG=(\VV,\EE)$ is called matching covered if $\forall e \in \EE(\GG), ~\exists \text{a perfect matching }M \text{ of }\GG \text{ s. t. } e \in M$. A cut $C$ of a matching covered graph $\GG$ is a separating cut if and only if, given any edge $e$, there is a perfect matching $M_{e}$ of $\GG$ such that $e \in M_{e}$ and $|C \cap M_{e}| = 1$. A cut $C$ in a matching covered graph $\GG$ is a tight cut of $\GG$ if $|C \cap M| = 1$ for every perfect matching $M$ of $\GG$. For, $X, Y \subseteq \VV(\GG)$, we denote the set of edges of $\EE(\GG)$ which have one endpoint in $X$ and the other endpoint in $Y$ by $E[X,Y]$. Let $\partial(X)=E[X,\overline{X}]$ be an edge cut, where $\overline{X}=\VV(\GG) \setminus X$. An edge cut is trivial if $|X|=1$ or $|\overline{X}|=1$. A matching covered graph, which is free of nontrivial tight cuts, is a brace if it is bipartite and is a brick if it is non-bipartite. An edge $e$ in a brace $\GG$ is \emph{thin} if, for every tight cut $\partial(X)$ of $\GG - e$, $|X| \le 3$ or $|\overline{X}| \le 3$.

Carvalho, Lucchesi and Murty conjectured that there exists a positive constant $c$ such that every brace $\GG$ has $c|\VV(\GG)|$ thin edges \cite{DBLP:journals/combinatorics/LucchesiCM15}. He and Lu \cite{HE2025153} showed a lower bound of thin edges in a brace in terms of the number of cubic vertices. We asked whether any planar brace exists that does not contain any cubic vertices. We answer negatively by showing that such set of planar braces is empty. We have been able to show a quantitively tight lower bound on the number of cubic vertices in a planar brace. We have proved tight upper bounds of nonthin edges and thin edges in a planar brace.

\end{abstract}
\section{Introduction}
We consider only simple graphs in this paper. For undefined notations and terminologies, we refer to \cite{Bondy2008GraphT}. Let $\GG$ be a graph with vertex set $\VV(\GG)$ and edge set $\EE(\GG)$. For two disjoint sets $A, B \subseteq \VV(\GG)$, we denote by $\GG[A,B]$ the bipartite graph with two colour classes $A$ and $B$. For a subset $X \subseteq \VV(\GG)$, let $\GG[X]$ denote the subgraph induced by $X$, and define the neighbourhood of $X$ as $N(X) = \{\, y \in \VV(\GG)\setminus X : \exists\, x \in X \text{ with } xy \in \EE(\GG) \,\}$. For a single vertex $x \in \VV(\GG)$, we abbreviate $N(x)$ for $N(\{x\})$. A vertex of degree $3$ is called a \emph{cubic vertex}. For subsets $X, Y \subseteq \VV(\GG)$, we write $\EE[X,Y]$ for the set of edges of $\GG$ with one endpoint in $X$ and the other in $Y$. For a subset $X$ of the vertex set $\VV(\GG)$ of a graph $\GG$, we denote the set of edges of $\GG$ which have exactly one end in $X$ by $\partial(X)$ and refer to it as the cut of $X$ or edge cut $\partial(X)$. For a vertex $v$ of $\GG$, we simplify the notation $\partial(\{ v \})$ to $\partial(v)$.

\begin{definition}[Matching]
    A matching is a set of pairwise non-adjacent edges. Formally, a matching of $\GG = (\VV,\EE)$ is a set $M \subseteq \EE(\GG)$ such that $|M \cap \partial(v)| \in \{ 0, 1 \} ~\forall v \in \VV(\GG)$.
\end{definition}

\begin{definition}[Perfect Matching]
    A matching $M$ of $\GG = (\VV,\EE)$ is called perfect if it covers all vertices of $\GG$ exactly once.
\end{definition}

\begin{definition}[Matching Covered Graph]
    A graph $\GG=(\VV,\EE)$ is called matching covered if $\forall e \in \EE(\GG), ~\exists \text{a perfect matching }M \text{ of }\GG \text{ s. t. } e \in M$. 
\end{definition}

\begin{definition}[Brace]\cite{GORSKY2023113249}
    \label{def:brace}
    A graph, other than the path of length three is a brace if it is bipartite and any two disjoint edges are part of a perfect matching.
\end{definition}

\begin{definition}[Cuts]\cite{10.1137/17M1138704}
    For a subset $X$ of the vertex set $\VV(\GG)$ of a graph $\GG$, we denote the set of edges of $\GG$ which have exactly one end in $X$ by $\partial(X)$ and refer to it as the cut of $X$.
\end{definition}
If $\GG$ is connected and $C\coloneqq \partial(X) = \partial(Y)$, then $Y=X$ or $Y=\overline{X}=\VV \setminus X$, and we refer to $X$ and $\overline{X}$ as the shores of $C$.
\par For a cut $C \coloneqq \partial(X)$ of a matching covered graph $\GG$, the parities of the cardinalities of the two shores are the same otherwise the order of the graph is not an even number. Here we shall only be concerned with those cuts that have shores of odd cardinality. A cut is trivial if either $|X| = 1$ or $|\overline{X}| = 1$ and is nontrivial otherwise.
\par Given any cut $C\coloneqq \partial(X)$ of a graph $\GG$, one can obtain a graph by shrinking $X$ to a single vertex $x$ (and deleting any resulting loops); we denote it by $\GG/\left( X \rightarrow x \right)$ and refer to the vertex $x$ as its contraction vertex. The two graphs $\GG/\left( X \rightarrow x \right)$ and $\GG/\left( \overline{X} \rightarrow \overline{x} \right)$ are the two $C$-contractions of $\GG$. When the names of the contraction vertices are irrelevant we shall denote the two $C$-contractions of $\GG$ simply by $\GG/X$ and $\GG/\overline{X}$.

\begin{definition}[Separating Cuts]\cite{10.1137/17M1138704}
    A cut $C\coloneqq \partial(X)$ of a matching covered graph $\GG$ is separating if both the $C$-contractions of $\GG$ are also matching covered. All trivial cuts are clearly separating cuts.
\end{definition}
The following proposition provides a necessary and sufficient condition under which a cut in a matching covered graph is a separating cut and can be easily proved.

\begin{proposition}
    \label{prop:sep_cut}
    A cut $C$ of a matching covered graph $\GG$ is a separating cut if and only if, given any edge $e$, there is a perfect matching $M_{e}$ of $\GG$ such that $e \in M_{e}$ and $|C \cap M_{e}| = 1$.
\end{proposition}

\begin{definition}[Tight cuts, bricks and braces]\cite{10.1137/17M1138704}
    A cut $C \coloneqq \partial(X)$ in a matching covered graph $\GG$ is a tight cut of $\GG$ if $|C \cap M| = 1$ for every perfect matching $M$ of $\GG$.
\end{definition}
It follows from \Cref{prop:sep_cut} that every tight cut of $\GG$ is also a separating cut of $\GG$. However, the converse does not always hold. For example, the cut shown in \Cref{fig:sep_but_not_tight} is a separating cut, but it is not a tight cut.
\begin{figure}[!htpb]
    \centering
    \includegraphics[scale=0.7]{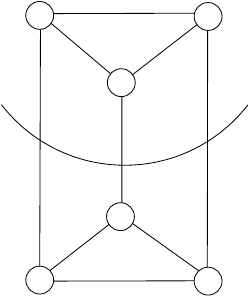}
    \caption{A cut which is separating but not tight}
    \label{fig:sep_but_not_tight}
\end{figure}
\par A matching covered graph, which is free of nontrivial tight cuts (i.e., for every tight cut $\partial(X)$, either $|X| = 1$ or $|\overline{X}|=1$), is a brace if it is bipartite and is a brick if it is non-bipartite. 

In particular, a cycle with four vertices $C_4$ is the minimum brace. An edge $e$ in a brace $\GG$ with at least six vertices is thin if, for every tight cut $\partial(X)$ of $\GG - e$, either $|X| \le 3$ or $|\overline{X}| \le 3$. Obviously, in a brace with six vertices, every edge is thin.

A \emph{brace} with at least six vertices is a $2$-extendable bipartite graph, i.e., every pair of non-adjacent edges can be extended to a perfect matching~\cite{LOVASZ1987187}. Braces are fundamental objects since they serve as the building blocks in decompositions of matching covered graphs.  
Lov\'{a}sz \cite{LOVASZ1987187} proved that any two applications of this decomposition process result in the same collection of \emph{bricks} (i.e., matching covered non-bipartite graphs without non-trivial tight cuts) and \emph{braces} (up to multiple edges). Many important problems in matching theory can be reduced to braces and bricks. Furthermore, if $\GG$ is a matching covered bipartite graph, then every tight cut decomposition of $\GG$ consists solely of braces \cite{DBLP:journals/combinatorics/LucchesiCM15}.

McCuaig \cite{McCuaig2001124} proved that every brace contains a \emph{thin edge}. Using thin edges, all braces can be generated from three fundamental classes of braces through several operations.  
Later, Carvalho, Lucchesi, and Murty \cite{DBLP:journals/combinatorics/LucchesiCM15} provided a simpler proof of the fact that every brace contains a thin edge, and further established that each brace contains at least two thin edges. In addition, they proposed the following conjecture.

\begin{conjecture} \cite{DBLP:journals/combinatorics/LucchesiCM15}
\label{conj:thn-edg-brc}
    There exists a positive constant $c$ such that every brace $\GG$ has $c|\VV(\GG)|$ thin edges.
\end{conjecture}

Let $\EE_{T}(\GG)$ denote the set of thin edges in a brace $\GG$. The authors He and Lu \cite{HE2025153} presented the following structural characterization of the nonthin edges in the subgraph induced by vertices with degree at least 4.

\begin{theorem} \cite{HE2025153}
\label{thm:frst-nthin-deg4}
Let $\GG$ be a brace and $S_{1}=\{ x \in \VV(\GG) \mid d(x) \ge 4 \}$. Then the graph spanned by $\EE\left( \GG[S_{1}] \right) \setminus \EE_{T}(\GG)$ is a forest.
\end{theorem}

Using \Cref{thm:frst-nthin-deg4}, He and Lu \cite{HE2025153} was able to show a lower bound of thin edges in a brace in terms of the number of cubic vertices which is stated in the following theorem.

\begin{theorem} \cite{HE2025153}
\label{thm:lwr-bnd-thn-edg-brc}
Let $\GG$ be a brace and let $n_3$ be the number of cubic vertices in $\GG$. If $n_3 = k|\VV(\GG)|$ and $k < 0.4$, then $\GG$ has at least $\frac{2-5k}{2}|\VV(\GG)| + 1$ thin edges.
\end{theorem}

\subsection{Properties of braces}
In this section, we discuss several properties of braces. The following characterization of braces will play a central role throughout this paper.

\begin{theorem} \cite{lovász1986matching}
\label{thm:brc-chrc-deg}
Let $\GG[ A, B]$ be a matching covered bipartite graph with at least six vertices. Then $\GG$ is a brace if and only if $|N(X)| \ge |X| + 2$, for every $X \subseteq A$ satisfying $1 \le |X| \le |A| - 2$.
\end{theorem}

It means that the degree of a vertex in braces is at least three by \Cref{thm:brc-chrc-deg}. The following corollary follows \Cref{thm:brc-chrc-deg} directly.

\begin{corollary} \cite{HE2025153}
\label{cor:x-intsct-b-sz}
Let $\GG[A, B]$ be a brace. Assume that $X \subset \VV(\GG)$, $|X \cap B| \le |B| - 2$ and $N(X \cap B) \subseteq X \cap A$. If $|X \cap A| = |X \cap B|$,
then $X=\emptyset$; if $|X \cap A| = |X \cap B| + 1$, then $X \cap B = \emptyset$.
\end{corollary}

Let $e$ be a nonthin edge of a brace $\GG$. Then there exists an edge cut $\partial(X)$ of $\GG$, such that $\partial(X)$ is a nontrivial tight cut in $\GG - e$, that is both $|X|$ and $|\overline{X}|$ are greater than 3 by the definition of nonthin edges; then $\partial(X)$ is an $S$-cut of $\GG$ associated with $e$. Note that there may exist more than one $S$-cuts of $\GG$ associated with $e$.

\begin{proposition} \cite{LOVASZ1987187}
\label{propo:tght-cut-charac}
Let $\GG[A, B]$ be a matching covered bipartite graph. Let $\partial(X)$ be an edge cut of $\GG$ such that $|X|$ is odd. Then $\partial(X)$ is tight if and only if the following statements hold. \begin{enumerate}
    \item $|| X \cap A | - | X \cap B|| = 1$;
    \item $\EE [ X \cap A, \overline{X} \cap B] = \emptyset$ if $|X \cap B| - | X \cap A | = 1$ and $\EE [ X \cap B, \overline{X} \cap A ] = \emptyset$ if $| X \cap A| - | X \cap B | = 1$.
\end{enumerate}
\end{proposition}

The following proposition of $S$-cuts follows \Cref{propo:tght-cut-charac} immediately.

\begin{proposition} \cite{HE2025153}
\label{propo:nonthin-scut}
Let $\GG[A, B]$ be a brace and let $uv$ be a nonthin edge of $\GG$. Assume that $\partial(X)$ is an $S$-cut of $\GG$ associated with $uv$, such that $u \in X \cap A$. Then $| X \cap A | = | X \cap B | - 1$ and $\EE [ X \cap A, \overline{X} \cap B] = \{ uv \}$.
\end{proposition}

\subsection{Structural Properties of Planar Bipartite Graphs}
A graph is \emph{planar} if it can be embedded in the Euclidean plane such that no two edges intersect except at a common vertex \cite{Bondy2008GraphT}. Such an embedding is called a planar embedding or a plane graph. A plane graph partitions the plane into a set of connected regions called faces. The number of vertices ($n$), edges ($m$), and faces ($f$) of any connected planar graph are related by a fundamental equation.

\begin{theorem}[Euler's Formula \cite{10.5555/22577}, \cite{Euler1736}]
For any connected planar graph with $n$ vertices, $m$ edges, and $f$ faces, the following relation holds: $n - m + f = 2$.
\end{theorem}

Euler's formula is a powerful tool for deriving bounds on the number of edges a planar graph can have. The key is to relate the number of faces to the number of edges. In any simple graph with at least 3 vertices, the boundary of each face in a planar embedding must consist of at least 3 edges. Since each edge borders at most two faces, a simple double-counting argument yields the inequality $3f \le 2m$. Substituting this into Euler's formula gives $n - m + (2m/3) \ge 2$, which simplifies to the well-known bound $m \le 3n - 6$ for any simple planar graph \cite{Cioab2009AFC}.

However, for bipartite graphs, a stronger constraint applies. A graph is bipartite if and only if it contains no cycles of odd length \cite{könig1936theorie}. This means the shortest possible cycle length, or \emph{girth}, in a bipartite graph with a cycle is at least 4. Consequently, in any planar embedding of a bipartite graph, the boundary of every face must be of length at least 4. This leads to a tighter inequality relating faces and edges: $4f \le 2m$, or $2f \le m$ \cite{Harary1969,scheinerman13,Tutte1963HowTD}. By substituting this improved bound into Euler's formula, we obtain a stricter limit on the edge density of planar bipartite graphs.

\begin{theorem} [\cite{10.5555/22577,Harary1969}]
\label{thm:edg-vrtx-bound}
Let $\GG$ be a simple, connected, planar bipartite graph with $n \ge 3$ vertices and $m$ edges. Then $m \le 2n - 4$.
\end{theorem}

\begin{proof}
From Euler's formula, $f = m - n + 2$. As argued above, the bipartite nature of $\GG$ implies that every face is bounded by at least 4 edges, leading to the inequality $4f \le 2m$. Substituting the expression for $f$ yields $4(m - n + 2) \le 2m$. This simplifies to $4m - 4n + 8 \le 2m$, which gives $2m \le 4n - 8$, and finally $m \le 2n - 4$.
\end{proof}

This inequality, $m \le 2n - 4$, is not merely a technical lemma; it represents a fundamental ``sparsity mandate'' for planar bipartite graphs. It has a direct consequence for the average degree of the graph, which is given by $2m/n$. From the inequality, we have $2m/n \le (4n - 8)/n = 4 - 8/n$. This shows that the average degree of any planar bipartite graph is strictly less than 4. This inherent sparsity is in stark contrast to the connectivity properties required by many graph classes, and as we will demonstrate, it creates an irreconcilable tension with the properties of a brace without cubic vertices.

\section{Main results}
He and Lu \cite{HE2025153} considers braces of at least six vertices where the minimum degree is three. A brace without cubic vertices is one in which every vertex has degree at least 4. That is, $n_3=0$ and $S_1 = \VV(\GG)$. This seemingly simple increase in the minimum degree requirement from 3 to 4 has profound structural consequences when combined with the constraints of planarity.

\begin{theorem}
\label{thm:brc-without-cubic}
The set of planar braces without any cubic vertices is empty.
\end{theorem}

\begin{theorem}
\label{thm:no-cubic-vrtx-plnr-brace}
Let $\GG$ be a planar brace of at least six vertices. Let $n_3$ be the number of cubic vertices in $\GG$. Then $n_3 \ge 8$.
\end{theorem}

By combining \Cref{thm:frst-nthin-deg4} with \Cref{thm:no-cubic-vrtx-plnr-brace}, we can place a strict, explicit upper bound on the number of nonthin edges in the subgraph induced by $S_1$.

\begin{corollary}
\label{cor:nonthin_bound}
Let $\GG$ be a planar brace on $n$ vertices. The subgraph spanned by the nonthin edges with both endpoints in $S_1$ is a forest containing at most $n - 9$ edges.
\end{corollary}

\begin{theorem}
\label{thm:thin_edge_analysis}
Let $\GG$ be a planar brace with $n_3$ cubic vertices, and assume $n_3 < 0.4|\VV(\GG)|$. The number of thin edges, $|\EE_T|$, is given by the bound: $|\EE_{T}| \ge |\VV(\GG)| - \frac{5}{2}n_3 + 1$. However, the planarity constraint $n_3 \ge 8$ imposes a ceiling on this lower bound. For any planar brace satisfying the condition $n_3 < 0.4|\VV(\GG)|$, the guaranteed number of thin edges is bounded by: $|\EE_{T}| \le |\VV(\GG)| - 19$.
\end{theorem}

\section{On the Existence of Planar Braces without Cubic Vertices}
\label{sec:exstnc-plnr-brc}
We must first establish whether the object of study, a planar brace without any cubic vertices, can exist. As foreshadowed in the introduction, the competing demands of the brace structure and planarity lead to a definitive negative conclusion.

\begin{proof} [Proof of \Cref{thm:brc-without-cubic}]
The proof proceeds by contradiction.
\begin{enumerate}
    \item \textbf{Assumption:} Let us assume that such a brace $\GG$ exists. By definition, $\GG$ is simultaneously a brace without cubic vertices and a planar graph.

    \item \textbf{Properties from the Brace Definition:}
    \begin{itemize}
        \item As a brace, $\GG$ is a connected bipartite graph \cite{HE2025153}.
        \item As a brace on more than 4 vertices, $\GG$ has a minimum vertex degree $\delta(\GG) \ge 3$ \cite{HE2025153}.
        \item The condition ``without any cubic vertices'' implies that $\GG$ has no vertices of degree 3. Combined with the minimum degree requirement for a brace, this means that every vertex in $\GG$ must have degree at least 4. Formally, $\delta(\GG) \ge 4$.
    \end{itemize}

    \item \textbf{Consequence of the Minimum Degree (Degree Sum Argument):} We consider the sum of the degrees of all vertices in $\GG$. By the handshaking lemma, this sum is equal to twice the number of edges, $m$:
    $\sum_{v \in \VV(\GG)} d(v) = 2m$. Since every vertex has a degree of at least 4, we can establish a lower bound on this sum: $\sum_{v \in \VV(\GG)} d(v) \ge \sum_{v \in \VV(\GG)} 4|\VV(\GG)| = 4n$. Combining these two expressions, we arrive at the inequality $2m \ge 4n$, which simplifies to $m \ge 2n$. This inequality is a direct consequence of $\GG$ being a brace without any cubic vertices. It mandates a certain level of edge density.

    \item \textbf{Consequence of Planarity:}
    \begin{itemize}
        \item As a planar brace, $\GG$ is a simple, connected, planar bipartite graph.
        \item As established in \Cref{thm:edg-vrtx-bound}, any such graph with $n \ge 3$ vertices must satisfy the upper bound on the number of edges derived from Euler's formula: $m \le 2n - 4$. This inequality is a direct consequence of $\GG$ being planar and bipartite. It imposes a strict limit on edge density, a ``sparsity mandate''.
    \end{itemize}

    \item \textbf{The Contradiction:} We have now derived two necessary, yet incompatible, conditions on the number of edges $m$ in our hypothetical brace $\GG$:
    \begin{itemize}
        \item From the cubic vertex less brace property: $m \ge 2n$.
        \item From the planar bipartite property: $m \le 2n - 4$.
    \end{itemize}
    Combining these inequalities yields: $2n \le m \le 2n - 4$. This chain of inequalities implies that $2n \le 2n - 4$, which simplifies to the patent absurdity $0 \le -4$.

    \par The assumption that a planar brace without any cubic vertices $\GG$ exists, leads to a logical contradiction. Therefore, the initial assumption must be false. No such brace can exist. Hence, the set of planar braces without any cubic vertices is empty. \qed
\end{enumerate} 
\end{proof}

This proof does more than simply answer a question of existence; it reveals a quantitative ``structural deficit''. The properties of a brace without any cubic vertices demand an ``edge budget'' of at least $2n$ edges to satisfy its minimum degree requirements. However, the properties of a planar bipartite graph impose a strict ``edge allowance'' of at most $2n - 4$ edges. The impossibility arises because the demand exceeds the allowance by a margin of at least 4. This perspective naturally leads to the next question: how can a graph be both a brace and planar? It can only do so by relaxing the minimum degree condition from 4 back to the brace with required minimum degree of 3. This requires the presence of cubic vertices, which lower the total degree sum and allow the brace to meet its planarity-imposed edge budget. The next section is dedicated to quantifying exactly how many such cubic vertices are necessary to resolve this structural deficit.

\section{A Structural Mandate: The Minimum Number of Cubic Vertices in Planar Braces}
\Cref{sec:exstnc-plnr-brc} establishes that no planar braces without any cubic vertices exist. However this non-existence result points us toward a more fundamental and well-posed structural question. The conflict between the degree requirements of a brace and the sparsity constraints of planarity is resolved by the presence of vertices of the lowest possible degree, which is 3. This naturally leads to the following question:

\textit{What is the minimum number of cubic vertices that any planar brace must possess?}

This question seeks to quantify the necessary structural compromise a brace must make to be embedded in the plane. It moves beyond a simple existence question to a problem of extremal graph theory: finding a sharp, non-trivial lower bound on a key structural parameter for an important class of graphs. The solution to this problem provides a new, fundamental theorem about the structure of all planar braces.

\subsection{Theorem Statement and Formal Proof}

We now present and prove \Cref{thm:no-cubic-vrtx-plnr-brace}, which establishes a tight lower bound on the number of cubic vertices in any planar brace.

\begin{proof}[Proof of \Cref{thm:no-cubic-vrtx-plnr-brace}]
Let $\GG = (\VV, \EE)$ be a planar brace with $n = |\VV(\GG)|$ vertices and $m = |\EE(\GG)|$ edges.
\begin{enumerate}
    \item \textbf{Partitioning the Vertex Set:} We partition the vertex set $\VV(\GG)$ based on the degrees of vertices. Let $\VV_3$ be the set of cubic vertices, so $n_3 = |\VV_3|$. Let $S_1$ be the set of noncubic vertices, which, for a brace, are those with degree 4 or greater. The size of this set is $|S_1| = n - n_3$.

    \item \textbf{Degree Sum Lower Bound:} We again use the handshaking lemma, $2m = \sum_{u \in \VV(\GG)} d(u)$. We can split this sum over our partitioned vertex set:
    $$2m = \sum_{u \in \VV_3} d(u) + \sum_{w \in S_1} d(w)$$
    By definition, every vertex in $\VV_3$ has degree exactly 3, and every vertex in $S_1$ has degree at least 4. This allows us to establish a lower bound on the sum of degrees:
    $$2m \ge 3 \cdot |\VV_3| + 4 \cdot |S_1|$$
    Substituting the expressions for the set sizes, we get:
    $$2m \ge 3n_3 + 4(n - n_3)$$
    Simplifying this expression gives us an inequality that captures the degree requirements of the brace structure in terms of its number of cubic vertices:
    $$2m \ge 3n_3 + 4n - 4n_3$$
    $$2m \ge 4n - n_3$$

    \item \textbf{Degree Sum Upper Bound:} As a planar brace, $\GG$ is a planar bipartite graph. From \Cref{thm:edg-vrtx-bound}, we have the strict upper bound on the number of edges: $m \le 2n - 4$. Multiplying by 2 gives an upper bound on the total degree sum that is dictated by planarity: $2m \le 4n - 8$.

    \item \textbf{Combining the Bounds:} We now have two inequalities bounding the same quantity, $2m$. One is a lower bound derived from the brace's internal degree structure, and the other is an upper bound derived from its external topological constraints. We can combine them into a single chain:
    $$4n - n_3 \le 2m \le 4n - 8$$

    \item \textbf{Deriving the Result:} By focusing on the leftmost and rightmost terms of this compound inequality, we obtain a relationship that is independent of the number of edges and vertices:
    \begin{align}
    4n - n_3 &\le 4n - 8 \nonumber \\
    \implies -n_3 &\le -8 \nonumber \\
    \implies n_3 &\ge 8 \nonumber
    \end{align}
\end{enumerate}
This completes the proof. Any graph that is both a brace and planar must contain at least eight cubic vertices. \qed
\end{proof}

This result is not merely a numerical curiosity; it is a sharp structural bound. A mathematical bound is most powerful when it is shown to be ``tight'', meaning there exists an object for which the inequality becomes an equality. To demonstrate the tightness of \Cref{thm:no-cubic-vrtx-plnr-brace}, we must exhibit a planar brace that has exactly eight cubic vertices.

Consider the graph of the cube, denoted $Q_3$.
\begin{itemize}
    \item $Q_3$ has $n=8$ vertices and $m=12$ edges.
    \item It is 3-regular (cubic), meaning every vertex has degree 3. Therefore, for $Q_3$, the number of cubic vertices is $n_3=8$.
    \item $Q_3$ is a well-known planar graph.
    \item $Q_3$ is bipartite.
    \item It is a known result in matching theory that $Q_3$ is a brace.
\end{itemize}

The cube graph $Q_3$ satisfies all the conditions of being a planar brace, and it has exactly $n_3=8$ cubic vertices. Its existence demonstrates that the lower bound established in \Cref{thm:no-cubic-vrtx-plnr-brace} cannot be improved. This elevates the theorem from a simple inequality to a characterization of an extremal property. It establishes the cube graph as a fundamental, minimal object in the study of planar braces, in the sense that any such graph must have at least as many cubic vertices as the cube itself.

Furthermore, for $Q_3$, the planarity bound is also met with equality: $m = 12$ and $2n-4 = 2 \times 8 - 4 = 12$. This shows that $Q_3$ is a maximal planar bipartite graph, one to which no more edges can be added without violating planarity or bipartiteness. This extremal nature in both its degree structure and its edge density underscores its significance as a canonical example in this domain.

\section{Upper bound on nonthin edges in a planar brace}

\begin{proof} [Proof of \Cref{cor:nonthin_bound}]
Let $\GG$ be a planar brace. By \Cref{thm:frst-nthin-deg4}, the subgraph $\GG_{NT}$ spanned by the nonthin edges within the vertex set $S_1$ is a forest. The number of vertices in this subgraph, $|\VV(\GG_{NT})|$, is a subset of $S_1$, so $|\VV(\GG_{NT})| \le |S_1|$. The number of edges in any forest is at most its number of vertices minus one. Therefore, the number of nonthin edges in $\GG$ with both endpoints in $S_1$ is bounded above by $|\EE(\GG_{NT})| \le |S_1| - 1$.

The set $S_1$ consists of all noncubic vertices. Thus, its size is $|S_1| = n - n_3$, where $n_3$ is the number of cubic vertices. By \Cref{thm:no-cubic-vrtx-plnr-brace}, any planar brace of at least six vertices must have $n_3 \ge 8$. This implies an upper bound on the size of $S_1$: $|S_1| = n - n_3 \le n - 8$. Substituting this into the edge bound for the forest gives:
$$|\EE(\GG_{NT})| \le |S_1| - 1 \le n - 8 - 1 = n - 9$$
Thus, the number of nonthin edges with both endpoints of degree four or more is strictly bounded above by $v-9$. \qed
\end{proof}

\begin{remark}
This corollary represents a significant sharpening of the original theorem for the planar case. While the original result is qualitative (the subgraph is a forest), our refinement is quantitative, leveraging the mandated sparsity of planar braces to impose a hard limit on the number of such nonthin edges.
\end{remark}

\section{Implications of planarity on the upper bound of thin edges in a planar brace}

\Cref{thm:lwr-bnd-thn-edg-brc} provides a powerful lower bound on the number of thin edges in a brace as a function of the proportion of its cubic vertices.
For general braces, the number of cubic vertices $n_3$ can be zero, allowing the lower bound of \emph{thin} edges to approach $|\VV(\GG)|+1$. However, for planar braces, the parameter $n_3$ is not free. The structural constraint from \Cref{thm:no-cubic-vrtx-plnr-brace} has a critical impact on the output of this formula of \Cref{thm:lwr-bnd-thn-edg-brc}. Specifically, if $n_3 = k|\VV(\GG)|$, the number of thin edges is at least $\frac{2-5k}{2}|\VV(\GG)|+1$, provided $k < 0.4$. Our result from \Cref{thm:no-cubic-vrtx-plnr-brace}, has a direct and significant implication for the application of this formula to the class of planar braces.

Our theorem establishes that for any planar brace, $n_3 \ge 8$. This imposes a new, size-dependent lower bound on the parameter $k$:
$$k = \frac{n_3}{n} \ge \frac{8}{n}$$
For general braces, the proportion of cubic vertices $k$ can be arbitrarily close to 0 (for example, in large complete bipartite braces $K_{m,n}$ where all vertices have degree much greater than 3). In such cases, as $k \to 0$, the He and Lu bound approaches $|\VV(\GG)|+1$.

However, for the class of planar braces, $k$ is bounded away from zero. This means that the most generous instances of the He and Lu bound (those for very small $k$) are not realized by planar braces. While \Cref{thm:lwr-bnd-thn-edg-brc} remains mathematically valid for any planar brace, its utility is shaped by this new constraint. For any specific planar brace on $n$ vertices, we know that the number of thin edges is bounded below by the He-Lu formula evaluated at $k=n_3/n$, where $n_3$ is at least 8.

This analysis also opens the door to a more refined perspective. The proof of \Cref{thm:lwr-bnd-thn-edg-brc} in \cite{HE2025153} is a counting argument that uses the result of \Cref{thm:frst-nthin-deg4} (the number of nonthin edges in $G$ is at most $|S_1|-1$) and combines it with a general degree sum calculation. A bespoke bound for planar braces could potentially be tighter by incorporating the powerful planarity constraint $m \le 2n - 4$ directly into this counting argument from the outset.

\begin{proof}[Proof of \Cref{thm:thin_edge_analysis}]
Let $\GG$ be a brace with $n=|\VV(\GG)|$ vertices. The formula $|\EE_T| \ge n - \frac{5}{2}n_3 + 1$ is taken directly from \Cref{thm:lwr-bnd-thn-edg-brc}. The function $f(n_3) = n - \frac{5}{2}n_3 + 1$ is a monotonically decreasing function of $n_3$. Therefore, the lower bound on $|\EE_T|$ is maximized when $n_3$ is minimized.

For the class of planar braces, \Cref{thm:no-cubic-vrtx-plnr-brace} establishes a rigid minimum of $n_3=8$. The number of cubic vertices in a bipartite graph must be even, so the smallest possible values for $n_3$ are 8, 10, 12, and so on.

By substituting the absolute minimum value $n_3=8$ into the function, we find the highest possible value that the lower bound can achieve for any planar brace:
\[ \therefore |\EE_{T}| \le n - \frac{5}{2}(8) + 1 = n - 20 + 1 = n - 19 \]
Since any planar brace must have $n_3 \ge 8$, the value of the lower bound $n - \frac{5}{2}n_3 + 1$ will always be less than or equal to $n-19$. \qed
\end{proof}

\begin{remark}
This finding reveals a fundamental divergence between the general theory of braces and the specific properties of planar braces. While \Cref{thm:lwr-bnd-thn-edg-brc} holds, its application to the planar case is permanently constrained. No planar brace can realize the higher-end lower bounds that are possible in general braces (e.g., for braces with $n_3 < 8$). The planarity condition creates a ``thin edge deficit'' relative to what is structurally possible in the broader, non-planar world, capping the guaranteed number of thin edges at a value significantly lower than the theoretical maximum.
\end{remark}

\section{Applicability of the He-Lu Theorems for the Planar Case}

Having established the non-existence of planar braces without any cubic vertices and derived a fundamental structural property that all planar braces must possess, we can now return to the core of the theorems and underlying lemma of He and Lu \cite{HE2025153}. We analyse their applicability and implications in light of our new findings.

\subsection{Validity of \Cref{thm:frst-nthin-deg4} in Planar Braces}

\Cref{thm:frst-nthin-deg4} from He and Lu states that for any brace $\GG$, the subgraph induced by the nonthin edges whose endpoints are all in $S_1$ (vertices of degree $\ge 4$) is a forest \cite{HE2025153}. The question is whether this theorem remains valid when the additional constraint of planarity is imposed on $\GG$. To answer this, we must examine the proof of the theorem.

The proof of \Cref{thm:frst-nthin-deg4}, as detailed in the paper and its analysis \cite{HE2025153}, rests on two preceding lemmas: Lemma 8 and Lemma 9 from \cite{HE2025153}.
\begin{itemize}
    \item \textbf{Lemma 8} establishes a property for a vertex $u \in S_1$ that is incident to two nonthin edges. It shows that the $S$-cuts associated with these edges, say $\partial(X)$ and $\partial(Y)$, must satisfy $|\overline{X} \cap \overline{Y}| \le 1$. The proof of this lemma relies on applications of \Cref{propo:nonthin-scut} and \Cref{cor:x-intsct-b-sz}. These propositions concern the cardinalities of vertex subsets in different parts of the bipartition and the structure of neighborhood sets related to $S$-cuts. The arguments are entirely combinatorial and set-theoretic, manipulating vertex set sizes based on the abstract definition of a brace.
    \item \textbf{Lemma 9} takes the property proven in Lemma 8 as a hypothesis. It assumes, for the sake of contradiction, that a cycle of nonthin edges exists within $\GG$. The proof then constructs a sequence of ``end sets'' associated with the edges of this cycle and shows that their cardinalities must form an infinite strictly decreasing sequence of positive integers ($|W_1| > |W_2| > \dots > |W_k| > |W_1^1| > \dots$). This is a logical impossibility, thus proving that no such cycle can exist. This argument, too, is purely combinatorial, relying on the properties of $S$-cuts and vertex sets.
\end{itemize}

Crucially, neither of these proofs makes any reference to or relies upon any geometric or topological properties of the brace $\GG$. The arguments are ``agnostic'' to whether the graph can be embedded in the plane, on a torus, or any other surface. The proofs depend only on the abstract combinatorial structure that defines a brace.

Therefore, imposing the additional constraint of planarity on $\GG$ does not invalidate any step in the logical chain of the proofs of Lemma 8, Lemma 9 from \cite{HE2025153}, or \Cref{thm:frst-nthin-deg4} itself. The theorem's conclusion is robust and holds for any subclass of braces, including planar braces.

Thus \Cref{thm:frst-nthin-deg4} from He and Lu holds for planar braces. The subgraph of a planar brace induced by its nonthin edges with endpoints of degree four or more is a forest. This conclusion highlights a key feature of many powerful results in structural graph theory: proofs that rely on fundamental combinatorial properties often have a broad and robust applicability across many different graph classes, irrespective of additional geometric constraints.

\section{Conclusion and Future Directions}
We have systematically investigated the structure of planar braces, beginning with a query about the applicability of recent theorems to the variants of planar braces. Our analysis has yielded several key contributions to the understanding of this graph class.

We provided a formal proof that the class of planar braces without any cubic vertices is empty. This was achieved by demonstrating an irreconcilable conflict between the minimum degree requirement of a brace having no cubic vertices ($\delta(G) \ge 4$) and the maximum average degree allowed for a planar bipartite graph (average degree $< 4$).

We determined the minimum number of cubic vertices in any planar brace. We proved a new theorem establishing this lower bound to be eight. We demonstrated that the bound $n_3 \ge 8$ is sharp by identifying the cube graph, $Q_3$, as a planar brace with exactly eight cubic vertices. This establishes $Q_3$ as a fundamental extremal graph in this context.

We have given quantitatively tight upper bounds on the number of nonthin edges and on the number of thin edges in a planar brace. 

We confirmed that Theorem 2 from He and Lu \cite{HE2025153}, which states that the induced subgraph of nonthin edges in $\GG$ is a forest, remains valid for planar braces. This is due to the purely combinatorial nature of its proof, which is unaffected by topological constraints.

We analysed the implications of our $n_3 \ge 8$ result on Theorem 3 from He and Lu \cite{HE2025153}. We showed that for planar braces, the proportion of cubic vertices $k$ is bounded away from zero by $k \ge 8/n$, which constrains the range of the thin-edge lower bound provided by their formula.

Our findings open several promising avenues for future research at the intersection of matching theory, topological graph theory, and combinatorics.

Our work shows that the cube graph $Q_3$ is an example of a planar brace that meets the bound $n_3=8$ with equality. A natural next step is to ask for a full characterization of this extremal family. Are all planar braces with exactly eight cubic vertices related to $Q_3$ through a set of graph operations? What other structural properties do they share?

The strict edge bound $m \le 2n - 4$ is unique to the plane. For a graph embedded on a torus (genus 1), the corresponding bound for a bipartite graph is $m \le 2n$. If we were to repeat the proof of \Cref{thm:no-cubic-vrtx-plnr-brace} for a toroidal brace, the final inequality would become $4n - n_3 \le 2m \le 4n$, which simplifies to $n_3 \ge 0$. This is a trivial bound. This dramatic difference suggests that the structural mandates for braces change significantly with the genus of the embedding surface. A systematic study of the minimum number of cubic vertices for braces on the torus, projective plane, and other surfaces could reveal a deep connection between topology and matching structure.

As noted by He and Lu \cite{HE2025153}, efficiently finding a thin edge in a general brace remains an open problem. Planar braces, however, possess additional structure that might be algorithmically exploitable. Their sparsity (guaranteed by $m \le 2n - 4$) and the mandated presence of at least eight cubic vertices could potentially be leveraged to design a more efficient algorithm for finding thin edges specifically within this subclass of braces.

Both braces and planar graphs play significant roles in the mathematical theory of structural rigidity \cite{graver1993combinatorial}. Planar graphs are central to Laman's theorem characterizing minimally rigid structures in the plane, and braces exhibit strong connectivity properties related to global rigidity. An investigation into the specific rigidity properties of planar braces and how the necessary presence of at least eight cubic vertices affects their stability and degrees of freedom could yield novel insights connecting matching theory, rigidity, and planar topology.

%

%
%
%
\bibliographystyle{splncs04}
\bibliography{mybibliography}
\end{document}